\documentclass[11pt]{amsart}
\usepackage{hyperref}
\usepackage{amssymb}
\usepackage{amsmath}
\usepackage{amsthm}
\usepackage{enumerate}
\usepackage{xcolor}
\theoremstyle{plain}
\newtheorem{thm}{Theorem}[section]
\newtheorem{prop}[thm]{Proposition}
\newtheorem{cor}[thm]{Corollary}
\newtheorem{lem}[thm]{Lemma}

\numberwithin{equation}{section}
\newtheorem{lemma}[equation]{Lemma}

\theoremstyle{definition}

\newtheorem{quest}[thm]{Question}

\newcommand{\ar}{\mathcal{R}}
\newcommand{\C}{\mathbb{C}}
\newcommand{\jury}{\diamond}
\numberwithin{equation}{section}
\title[Oppenheim--Schur inequalities for causal products]{Oppenheim--Schur inequalities for causal products}
\author[D.~Guillot, J.~Mashreghi, \and P.K.~Vishwakarma]{Dominique Guillot, Javad Mashreghi, \\ \and Prateek Kumar Vishwakarma}
\address[D. Guillot]{Department of Mathematical Sciences, University of Delaware, Newark, DE, USA 19716}
\email{dguillot@udel.edu}
\address[J. Mashreghi]{D\'epartement de math\'ematiques et de statistique, Universit\'e Laval, Qu\'ebec, QC, Canada G1V 0K6.}
\email{javad.mashreghi@ulaval.ca}
\address[P.K. Vishwakarma]{D\'epartement de math\'ematiques et de statistique, Universit\'e Laval, Qu\'ebec, QC, Canada G1V 0K6.}
\email{prateek-kumar.vishwakarma.1@ulaval.ca,~prateekv@alum.iisc.ac.in}
\date{\today}
\begin{document}

\begin{abstract}
We establish a class of Oppenheim--Schur-type inequalities for the convolutional Jury product of positive semidefinite matrices. These results extend to a causal convolutional setting the classical Schur and Oppenheim inequalities associated with the Hadamard product. Our approach highlights structural parallels between entrywise and convolution-based matrix operations, revealing how positivity constraints interact with causality. Building on this perspective, we introduce a broader family of causal matrix products and prove unified inequalities that simultaneously recover the classical Schur and Oppenheim bounds as well as their convolutional Jury counterparts. These results provide a common framework for understanding positivity-preserving matrix products and suggest further connections between classical matrix analysis and causal operator structures.

\end{abstract}

\maketitle

\section{Introduction}

For an integer \( N \geq 1 \), a matrix \( A \in \mathbb{C}^{N \times N} \) is said to be \textit{positive definite} if \( z^* A z > 0 \) for all nonzero vectors \( z \in \mathbb{C}^N \). Similarly, \( A \) is called \textit{positive semidefinite} if \( z^* A z \geq 0 \) for all \( z \in \mathbb{C}^N \). Positive semidefinite matrices exhibit a variety of remarkable properties, many of which are reflected in classical inequalities. One of the earliest and most celebrated of these is due to J.~Hadamard \cite{hadamard1893resolution}, given by
\begin{equation*}
\det A \leq \prod_{k=0}^{N-1}a_{kk}    
\end{equation*}
for all positive semidefinite $A=(a_{jk})_{j,k=0}^{N-1}$. This result was later extended by E.~Fischer in \cite{fischer1908uber} to a more general determinantal bound. In seminal papers, a series of inequalities that further strengthened these bounds were proved by A.~Oppenheim and I.~Schur~\cite{MR1574213,schur1911}. These involve the \textit{Hadamard (Schur) product}, defined by 
\[
A\circ B:=(a_{jk}b_{jk})_{j,k=0}^{N-1}
\]
for $A=(a_{jk})_{j,k=0}^{N-1}$ and $B=(b_{jk})_{j,k=0}^{N-1}$. The classical Schur Product Theorem \cite{schur1911} states that if $A$ and $B$ are positive semidefinite, then their Hadamard product $A\circ B$ is also positive semidefinite. This semigroupoid structure has led to a number of developments in recent years, with applications and connections across several areas of mathematics; see, for instance, \cite{belton2019panorama,guillot2025positivity, khare2022matrix} and the references therein. In this paper, we focus on one concrete consequence of this structure: since $A\circ B$ is positive semidefinite whenever $A,B$ are positive semidefinite, we have, in particular $\det (A\circ B) \ge 0$. The work of Oppenheim and Schur sharpens this basic observation by showing that, for positive semidefinite $A$ and $B$, the following chain of increasingly strong inequalities holds:
\begin{align}\label{Oppen}
\begin{aligned}
\det (A\circ B) &\geq \det (A B) \\
\det (A\circ B) &\geq \det A \prod_{k=0}^{N-1} b_{kk}\\
\det (A\circ B) + \det (A B) &\geq \det A \prod_{k=0}^{N-1} b_{kk} + \det B \prod_{k=0}^{N-1}a_{kk}
\end{aligned}
\end{align}

The purpose of this paper is to establish analogous inequalities for other, more recently discovered, matrix products that preserve positivity. These include the matrix convolution, also known as the \textit{Jury product} (whose positivity related properties were explored in \cite{jurythesis, convolution-pres}) as well as for more general structured classes of products introduced in \cite{guillot2026sharp}. We develop a systematic treatment of these products and their corresponding determinant inequalities.

\section{The Jury product and inequalities}

Henceforth, fix an integer $N\geq 1$, and follow the convention that the entries of all matrices $A=(a_{jk})\in \C^{N\times N}$ are indexed by the row index $j=0,1,\dots,N-1$ and the column index $k=0,1,\dots,N-1$.\medskip

The \textit{Jury product} of $A=(a_{jk})$ and $B=(b_{jk})$, denoted by $A\jury B\in \C^{N\times N}$, is defined by
\begin{align}\label{Jury-product}
(A\jury B)_{jk}:=\sum_{m=0}^{j}\sum_{n=0}^{k}a_{m,n}b_{j-m,k-n}
\end{align}
for all $j=0,1,\dots,N-1$ and $k=0,1,\dots,N-1$. \medskip

This product was introduced by M.~Jury in his Ph.D.~thesis~\cite{jurythesis} as part of a unified approach to the classical Carathéodory--Fejér and Nevanlinna--Pick interpolation theorems in complex function theory, formulated as extension problems for positive matrices. For example, the Carathéodory interpolation problem \cite{Car1907}, studied by Schur \cite{schur1917,schur1918}, seeks a holomorphic function on the open unit disk whose Taylor expansion begins with a prescribed sequence of complex numbers and whose modulus is bounded by one throughout the disk. To support his unified treatment, Jury introduced the aforementioned product, defined with respect to a partial order on a fiber bundle, and used it to solve the classical interpolation problems. His framework also lays the groundwork for studying generalized forms of these problems. See Jury's thesis~\cite{jurythesis}, the work of Dritschel--Marcantognini--McCullough \cite{dritschel2007}, and the book by Agler and McCarthy~\cite{agler2002pick} for more details.

    Beyond their striking similarities in applications to interpolation problems, the Schur and Jury products share an important structural property: like the Schur product, the Jury product endows the \textit{closed convex cone} of positive semidefinite matrices with a semigroup structure. More precisely,  Jury showed that \medskip

\begin{center}
{\it if $A$ and $B$ are positive semidefinite, then so is $A\jury B$.}    
\end{center}\medskip

\noindent See \cite[Lemma~5.32]{agler2002pick} for details. As for the Schur product, our first main result shows that a stronger inequality holds for the Jury product.

\begin{thm}[The inequality]\label{T:main}
For any integer $N\geq 2$, we have
\begin{align*}
\det (A\jury B) &\geq b_{00}^{N} \det A  + a_{00}^{N}\det B
\end{align*}
for all positive semidefinite $A=(a_{jk}),B=(b_{jk})\in \C^{N\times N}$.
\end{thm}

The clear similarities and differences of the above with Oppenheim--Schur's \eqref{Oppen} are evident. In particular, the additional nonnegative term $\det(AB)$ is \textit{not} required for the Jury product. Also, the first entries $a_{00}$ and $b_{00}$ are highlighted with degrees $N$ each, compared to all diagonal entries in Oppenheim--Schur's \eqref{Oppen} with degrees $1$ each. Furthermore, under the hypothesis of Theorem~\ref{T:main}, the inequality
\begin{align}\label{Opp:special}
\det (A\jury B) &\geq b_{00}^{N} \det A
\end{align}
follows as a special case. Although this constitutes a special case, it remains valuable for both \textit{technical} and \textit{historical} reasons. On the technical side, its direct proof highlights fundamental properties of the Jury product and offers valuable context for the more intricate general proof. From a historical perspective, this special case can be viewed as an analog of the second inequality in \eqref{Oppen}, which was, in fact, originally shown by Oppenheim~\cite{MR1574213}.

\subsection{Organization of the paper} The remainder of the paper is structured as follows. First, we prove the special \eqref{Opp:special} in Section~\ref{S:special-case}. Then we pause in Section~\ref{S:transformation} to examine the properties of a variant of a transformation used by Oppenheim \cite{MR1574213}. This digression plays a crucial role in proving the main inequality in Section~\ref{S:gen-proof}. Then in the final Section~\ref{S:causal} we introduce causal products of matrices, and extend and unify the classical and our Oppenheim--Schur inequalities.

\section{The special case}\label{S:special-case}

As before, for a fixed integer $N\geq 1$, we follow the convention that the entries of matrices $A=(a_{jk})\in \C^{N\times N}$ are indexed by $j,k=0,1,\dots,N-1$, and the entries of vectors $v=(v_0,\dots,v_{N-1})^T\in \C^{N}$ are indexed by $j=0,1,\dots,N-1$.

\begin{lemma}\label{LJuryRank1}
Suppose $u=(u_0,\dots,u_{N-1})^T,v=(v_0,\dots,v_{N-1})^T \in \C^{N}$, and let $A = uu^*$ and $B = vv^*$. Then
\[
A \jury B = (u \jury v)(u \jury v)^*, 
\]
where $(u \jury v)^T \in \C^{N}$ with $(u \jury v)_j = \sum_{m=0}^j u_m v_{j-m}$. 
\end{lemma}
\begin{proof}
We have 
\begin{align*}
(A \jury B)_{jk} &= \sum_{m=0}^j \sum_{n=0}^k a_{mn} b_{j-m,k-n} = \sum_{m=0}^j \sum_{n=0}^k u_m \overline{u_n} v_{j-m} \overline{v_{k-n}} \\
&= \left(\sum_{m=0}^j u_m v_{j-m}\right)\overline{\left(\sum_{n=0}^k u_n v_{k-n}\right)} = (u \jury v)_j \overline{(u \jury v)_k}. 
\end{align*}
\end{proof}

\noindent Given $v=(v_0,\dots,v_{N-1})^T \in \C^{N}$, let $C_v \in \C^{N\times N}$ be the matrix 
\[
C_v := \begin{pmatrix}
    v_0 & 0 & 0 & \dots & 0 \\
    v_1 & v_0 & 0 & \dots & 0 \\
    v_2 & v_1 & v_0 & \dots & 0 \\
    \vdots \\
    v_{N-1} & v_{N-2} & v_{N-3} & \dots & v_0
\end{pmatrix}.
\]
Observe that, for $u, v \in \C^{N}$, we have $u \jury v = C_v u$ and $\det C_v = v_0^{N}$.

\begin{lemma}\label{LJuryDetRank1}
Let $A \in \C^{N\times N}$ be positive semidefinite, and let $v=(v_0,\dots,v_{N-1})^T \in \C^{N}$. Then 
\begin{align*}
\det(A \jury vv^*) = |v_0|^{2N} \det A. 
\end{align*}
\end{lemma}
\begin{proof}
Writing $A$ as a Gram matrix, we have 
\[
A = \sum_{j=0}^{N-1} u_j u_j^*
\]
for some $u_0, \dots, u_{N-1} \in \C^{N}$. Using Lemma \ref{LJuryRank1}, 
\begin{align*}
A \jury vv^* &= \left(\sum_{j=0}^{N-1} u_j u_j^*\right) \jury vv^* = \sum_{j=0}^{N-1} (u_j \jury v) (u_j \jury v)^* = \sum_{j=0}^{N-1} (C_v u_j)(C_v u_j)^*\\
&= \sum_{j=0}^{N-1} C_v u_j u_j^* C_v^* = C_v \left(\sum_{j=0}^{N-1} u_j u_j^*\right) C_v^* = C_v A C_v^*.
\end{align*}
Thus, $\det (A \jury vv^*) = \det(C_v) \det(A) \det(C_v^*) = \det A \cdot |v_0|^{2N}$.
\end{proof}

Recall that for two positive semidefinite matrices $A, B \in \C^{N\times N}$, we have $\det(A+B) \geq \det(A) + \det(B)$. Using this and Lemma \ref{LJuryDetRank1}, we obtain an Oppenheim-type inequality for Jury's product. 
\begin{thm}
Let $A=(a_{jk})$ and $B=(b_{jk}) \in \C^{N\times N}$ be positive semidefinite. Then 
\[
\det(A \jury B) \geq b_{00}^{N} \det A. 
\]
Moreover, equality is achieved if $B$ has rank $1$.
\end{thm}
\begin{proof}
Recall that $B$ admits a Cholesky decomposition $B = LL^*$ where $L$ is lower triangular \cite[Corollary 7.2.9]{horn2013matrix}. Let $v_0, \dots, v_{N-1} \in \C^{N}$ be the columns of $L$. Using the lemma \ref{LJuryDetRank1} and the super-additivity of the determinant on positive semidefinite matrices, we obtain: 
\begin{align*}
\det(A \jury B) = \det(A \jury \sum_{j=0}^{N-1} v_j v_j^*) \geq \sum_{j=0}^{N-1} \det(A \jury v_j v_j^*)&= \sum_{j=0}^{N-1} |(v_j)_0|^{2N} \det A\\ &= b_{00}^{N} \det A.     
\end{align*}
That we have equality when $B$ has rank $1$ is Lemma \ref{LJuryDetRank1}.
\end{proof}

\section{The transformation $\widetilde{A}$}\label{S:transformation}

Fix the integer $N\geq 1$, and let $A=(a_{jk})_{j,k=0}^{N-1}\in \C^{N\times N}$. To apply an inductive argument, we let $A_{n}=(a_{jk})_{j,k=0}^{n}$, where $n=0,1,\dots,N-1$, denote the leading principal submatrices of $A$. Hence, with this notation, $A_{N-1}=A$ and $A_0=a_{00}$. We define $\widetilde{A} = (\widetilde{a}_{jk})_{j,k=0}^{N-1}$ via

\begin{align}\label{EAstar}
\widetilde{a}_{jk} = \left\{
\begin{array}{ccc}
a_{jk} - \frac{\det(A_{N-1})}{\det(A_{N-2})} & \mbox{if} & j = k = N-1,\\
& & \\
a_{jk} &  & \mbox{otherwise.}\\
\end{array} \right.
\end{align}

Of course, in the above transformation from $A$ to $\widetilde{A}$, we implicitly assumed that $\det(A_{N-2})\ne 0$. Such transformations have a rich history in matrix analysis in order to deduce matrix identities or inequalities. See, for example, the lemma on page two of \cite{MR1574213}. See also Definition 2 in \cite{MR2344681}. The main feature of this transformation is the following result, whose proof is folklore; for convenience, we provide a sketch.

\begin{lem} \label{L:A-Astar}
Let $A \in \C^{N \times N}$ be positive definite. Then $\widetilde{A}$ is positive semidefinite.
\end{lem}

\begin{proof}
Since $A_n = (\widetilde{A})_n$, for $n=0,1,\dots,N-2$, by the Sylvester criterion \cite[Theorem 7.2.5]{horn2013matrix}, we have $\det(\widetilde{A}_n) >0$. Moreover, upon expansion with respect to the last column, we see that
\[
\det((\widetilde{A})_{N-1}) = \det(\widetilde{A}) = \det(A) - \frac{\det(A_{N-1})}{\det(A_{N-2})} \det(A_{N-2}) = 0.
\]
Hence, $\widetilde{A}$ is positive semidefinite by \cite[Theorem 7.2.5(c)]{horn2013matrix}.
\end{proof}

Note that under the hypothesis of Lemma \ref{L:A-Astar}, it follows that
\begin{equation}\label{E:ineq-ann}
a_{(N-1),(N-1)} \geq \frac{\det(A_{N-1})}{\det(A_{N-2})},
\end{equation}
which is again a known result.

\section{The general proof}\label{S:gen-proof}

\begin{thm} \label{T:inequality}
For integer $N\geq 1$ and positive definite $A=(a_{jk}),B=(b_{jk})\in \C^{N\times N}$, we have
\begin{equation}\label{E:Jury-3}
\det(A \diamond B) \geq a_{00}b_{00} \prod_{n=1}^{N-1} \left( a_{00} \frac{\det(B_n)}{\det(B_{n-1})} + b_{00} \frac{\det(A_n)}{\det(A_{n-1})} \right).
\end{equation}
\end{thm}

\begin{proof}
Set $C=A \diamond B$, and define $A_n$, $B_n$, and $C_n$ as in Section \ref{S:transformation}. Then we know from Jury's theorem that $\widetilde{A}\jury \widetilde{B}$ is positive semidefinite, which upon taking the determinant and a simple expansion yields

\begin{equation}\label{E:Jury-2}
\det(C_{N-1}) \geq \left( a_{00} \frac{\det(B_{N-1})}{\det(B_{N-2})} + b_{00} \frac{\det(A_{N-1})}{\det(A_{N-2})} \right) \det(C_{N-2}).
\end{equation}
By induction, we multiply \eqref{E:Jury-2} by the values $N-1, N-2, \dots,0$ and obtain \eqref{E:Jury-3}.
\end{proof}

We are now ready for the final proof.

\begin{proof}[Proof of Theorem~\ref{T:main}]
Replacing $A$ and $B$ respectively by $A+\varepsilon I$ and $B+\varepsilon I$, where $\varepsilon>0$, the inequality \ref{E:Jury-3}, in particular, implies
\[
\det((A+\varepsilon I) \diamond (B+\varepsilon I)) \geq (a_{00}+\varepsilon)^{N} \det(B+\varepsilon I) + (b_{00}+\varepsilon)^{N} \det(A+\varepsilon I).
\]
The result follows by letting $\varepsilon \to 0^+$.
\end{proof}

\begin{quest}
We conclude the section with a question. It is straightforward to verify that Theorem~\ref{T:main} holds as an identity for $N=2$. Thus, for \( N \geq 3 \), it may be of interest to construct positive definite \( A \) and \( B \) for which equality holds in Theorem~\ref{T:main}. In contrast to Oppenheim’s inequality~\eqref{Oppen}, one can verify that strict inequality holds for all positive definite diagonal matrices \( A \) and \( B \), suggesting that the proposed question may not be trivial.
\end{quest}

\section{The inequalities for causal products}\label{S:causal}

In this final section, we show how the classical Oppenheim--Schur inequalities and our inequalities for the Jury product can be unified and extended to a much broader family of matrix products, which we call \textit{causal products}. Throughout the rest of this section, for a given integer $N\geq 1$, and for every $n=0, \dots, N-1$, we fix:
\begin{enumerate}
\item a set $T_n\subseteq \{0,\dots,n\}$ with $n\in T_n$; 
\item a permutation $\sigma_{n}:T_n\longrightarrow T_n$.
\end{enumerate}
For complex matrices $A=(a_{ij})_{i,j=0}^{N-1},B:=(b_{ij})_{i,j=0}^{N-1}$, we define a \textit{causal product} $A\star B\in \mathbb{C}^{N\times N}$ through
\begin{align}\label{eq:causal}
    (A\star B)_{jk}:=\sum_{(p,q)\in T_j\times T_k} a_{p,q}b_{\sigma_{j}(p),\sigma_{k}(q)}
\end{align}
for all $0\leq j,k\leq N-1$. Notice that for special choices of $(T_n)_{n}$ and $(\sigma_n)_n$, we can obtain the Hadamard and Jury products:
\begin{align}\label{eq:Hadamard+Jury-special-cases}
\begin{aligned}
T_n:=\{n\} \quad & \textrm{ and } \quad \sigma_n: n \mapsto n \qquad  \qquad \mbox{(Hadamard)}\\
T_n:=\{0,\dots,n\} \quad & \textrm{ and } \quad \sigma_n: j \mapsto n-j \qquad  \qquad \mbox{(Jury)}
\end{aligned}
\end{align}

Furthermore, as for the Schur product and the Jury product, the causal products $\star$ preserve positivity.

\begin{prop}
If $A,B \in \C^{N \times N}$ are positive semidefinite, then $A\star B$ is positive semidefinite.
\end{prop}
\begin{proof}
Endow each finite dimensional complex vector space with the standard Hilbert--Schmidt inner product $\langle \cdot,\cdot \rangle$ that is linear in the first variable and conjugate linear in the second. Since $A,B$ are positive semidefinite, there exist $v_0,\dots,v_{N-1}\in \mathbb{C}^{N}$ and $u_0,\dots,u_{N-1}\in \mathbb{C}^{N}$ such that each $a_{jk}=\langle v_{k},v_{j}\rangle$ and $b_{jk}=\langle u_{k},u_{j}\rangle$. Then   
\begin{align*}
    (A\star B)_{jk}&=\sum_{(p,q)\in T_j\times T_k} a_{p,q}b_{\sigma_{j}(p),\sigma_{k}(q)}
    = \sum_{(p,q)\in T_j\times T_k} \langle v_{q},v_{p}\rangle \langle u_{\sigma_{k}(q)},u_{\sigma_j(p)}\rangle\\
    &= \sum_{(p,q)\in T_j\times T_k} \langle v_{q}\otimes u_{\sigma_{k}(q)},v_{p}\otimes u_{\sigma_j(p)}\rangle \\
    &= \Big{\langle} \sum_{q\in T_k} v_{q}\otimes u_{\sigma_{k}(q)}, \sum_{p\in T_j}v_{p}\otimes u_{\sigma_j(p)}\Big{\rangle},
\end{align*}
where $\otimes$ is the standard Kronecker product. Therefore, $A\star B$ is a Gram matrix and is therefore positive semidefinite.
\end{proof}

Since Schur and Jury products are special cases of causal products, it is natural to ask whether the Schur--Oppenheim inequalities extend to such products. The next result establishes the corresponding inequalities for certain special causal products.

\begin{thm}\label{main:causal-2}
Suppose $\star$ is a causal product as in Equation \eqref{eq:causal}. Then, for all positive semidefinite matrices $A, B \in \C^{N \times N}$, 
\begin{enumerate}
\item If $\sigma_n(n) = n$ for all $n=0,\dots,N-1$, then 
\[
\det A\star B + \det AB \geq \det A\prod_{n=0}^{N-1}b_{nn} + \det B \prod_{n=0}^{N-1}a_{nn}.
\]
\item If $\sigma_n(n) \ne n$ for all $n=0,\dots,N-1$, then 
\[
\det A\star B \geq \det A\prod_{n=0}^{N-1}b_{\sigma_n(n),\sigma_n(n)} + \det B \prod_{n=0}^{N-1}a_{\sigma_n^{-1}(n),\sigma_n^{-1}(n)}.
\]
\end{enumerate}
\end{thm}
\begin{proof}
    The result follows from Corollaries \ref{cor:causal} and \ref{cor:causal-1}, which are proved in the following.
\end{proof}

Notice that in the special cases of $T_n, \sigma_n$ given by Equation \eqref{eq:Hadamard+Jury-special-cases}, Theorem~\ref{main:causal-2}, recover the Schur--Oppenheim inequalities for the Hadamard and Jury products. Theorem~\ref{main:causal-2} thus unifies these inequalities for two products that both originated in applications to interpolation problems in complex function theory \cite{agler2002pick}. 

We now proceed to prove Theorem~\ref{main:causal-2}. The first step parallels Oppenheim’s original argument in \cite{MR1574213}, using an inductive procedure to obtain a lower bound for $\det (A\star B)$ in terms of products of functions associated with the leading principal submatrices. This connects the causal nature of the $\star$ products with the leading principal minors. 

Recall that for a matrix $A = (a_{jk})_{j,k=0}^{N-1} \in \C^{N \times N}$, we let $A_{n}=(a_{jk})_{j,k=0}^{n}$ denote the leading principal submatrices of $A$, where $n=0, \dots, N-1$. Given a set $S$, we also let ${\bf 1}_S$ denote the indicator function of $S$.

\begin{thm}\label{T:causal} 
For positive definite $A,B \in \C^{N \times N}$, and for any causal product $\star$ as in Equation \eqref{eq:causal}, we have
\begin{align*}
\det A\star B \geq & ~a_{00}b_{00} \prod_{n=1}^{N-1}\bigg{(} \frac{\det A_n}{\det A_{n-1}}b_{\sigma_{n}(n),\sigma_{n}(n)} + \frac{\det B_n}{\det B_{n-1}}a_{\sigma_{n}^{-1}(n),\sigma_{n}^{-1}(n)} \\
&\qquad \qquad \qquad -\frac{\det A_n}{\det A_{n-1}}\frac{\det B_n}{\det B_{n-1}} {\bf 1}_{\{\sigma_{n}(n)\}}(n)\bigg{)}.
\end{align*}
\end{thm}
\begin{proof}
Following the steps and notation in Theorem~\ref{T:inequality}, if $j=k=N-1,$ we have that
\begin{align*}
(\widetilde{A}\star \widetilde{B})_{jk}=(A\star B)_{N-1,N-1} &- \frac{\det A_{N-1}}{\det A_{N-2}}b_{\sigma_{N-1}(N-1),\sigma_{N-1}(N-1)}\\
& - \frac{\det B_{N-1}}{\det B_{N-2}}a_{\sigma_{N-1}^{-1}(N-1),\sigma_{N-1}^{-1}(N-1)} \\ 
&+\frac{\det A_{N-1}}{\det A_{N-2}}\frac{\det B_{N-1}}{\det B_{N-2}} {\bf 1}_{\{\sigma_{N-1}(N-1)\}}(N-1),
\end{align*}
and otherwise $(\widetilde{A}\star \widetilde{B})_{jk}=(A\star B)_{jk}$. Therefore, the following is nonnegative:
\begin{align*}
\det (\widetilde{A}\star \widetilde{B}) =&\det (A\star B)_{N-1} -\bigg{[}  \frac{\det A_{N-1}}{\det A_{N-2}}b_{\sigma_{N-1}(N-1),\sigma_{N-1}(N-1)} \\
& + \frac{\det B_{N-1}}{\det B_{N-2}}a_{\sigma_{N-1}^{-1}(N-1),\sigma_{N-1}^{-1}(N-1)} \\
& -\frac{\det A_{N-1}}{\det A_{N-1}}\frac{\det B_{N-1}}{\det B_{N-2}} {\bf 1}_{\{\sigma_{N-1}(N-1)\}}(N-1)\bigg{]} \det(A\star B)_{N-2}.
\end{align*}
Thus
\begin{align}\label{eq:intermediate}
\begin{aligned}
&\det (A\star B)_{N-1} \\
&\geq \bigg{[} \frac{\det A_{N-1}}{\det A_{N-2}}b_{\sigma_{N-1}(N-1),\sigma_{N-1}(N-1)}  + \frac{\det B_{N-1}}{\det B_{N-2}}a_{\sigma_{N-1}^{-1}(N),\sigma_{N-1}^{-1}(N-1)} \\
&  \qquad -\frac{\det A_{N-1}}{\det A_{N-2}}\frac{\det B_{N-1}}{\det B_{N-2}} {\bf 1}_{\{\sigma_{N-1}(N-1)\}}(N-1)\bigg{]} \det(A\star B)_{N-2}.
\end{aligned}
\end{align}
Now use induction to obtain the final inequality.
\end{proof}

\begin{cor}\label{cor:causal}
If a causal product $\star$ is such that ${\bf 1}_{\{\sigma_{n}(n)\}}(n)=0$ for all $n\geq 1$, then for positive semidefinite matrices $A,B \in \C^{N \times N}$, we have
\begin{align*}
    \det A\star B  \geq \det A\prod_{n=0}^{N-1}b_{\sigma_n(n),\sigma_n(n)} + \det B \prod_{n=0}^{N-1}a_{\sigma_n^{-1}(n),\sigma_n^{-1}(n)}.
\end{align*}
\end{cor}
\begin{proof}
The proof is similar to the proof of Theorem~\ref{T:main} via Theorem~\ref{T:causal}.
\end{proof}

In the case of the Jury product, we have $\sigma_{n}(n)=\sigma_n^{-1}(n)=0$ and Corollary~\ref{cor:causal} recovers Theorem~\ref{T:main}.

\begin{cor}\label{cor:causal-1}
If a causal product $\star$ is such that ${\bf 1}_{\{\sigma_{n}(n)\}}(n)=1$ for all $n\geq 1$, i.e., $\sigma_{n}(n)=n$ for all $n\geq 1$, then:
\begin{align*}
    \det A\star B + \det AB \geq \det A\prod_{n=0}^{N-1}b_{nn} + \det B \prod_{n=0}^{N-1}a_{nn},
\end{align*}
for positive semidefinite $A$ and $B$.
\end{cor}
\begin{proof}
We prove the inequality for positive definite matrices $A,B \in \C^{N \times N}$. By continuity, the same result holds for all positive semidefinite matrices. 

For $n= 0,1,\dots,N-1$, define the following expressions:
\begin{align*}
\ell_n & := \det (A\star B)_n + \det A_n \det B_n - \det B_{n}\prod_{j=0}^{n} a_{jj}  -  \det A_{n}\prod_{j=0}^{n} b_{jj},    \\
L_{n} & := \frac{\det A_{n}}{\det A_{n-1}}b_{n,n} + \frac{\det B_{n}}{\det B_{n-1}}a_{n,n} -\frac{\det A_{n}}{\det A_{n-1}}\frac{\det B_{n}}{\det B_{n-1}}.
\end{align*}
Then, from \eqref{eq:intermediate} we have
\begin{align*}
\det (A\star B)_{N-1} \geq L_{N-1} \det(A\star B)_{N-2}.
\end{align*}
Now, writing $\det (A\star B)_{N-1}$ and $\det (A\star B)_{N-2}$ in terms of $\ell_{N-1}$ and $\ell_{N-2}$, we obtain
\begin{align}\label{eq:startingpoint}
\begin{aligned}
& \ell_{N-1} - \det A_{N-1} \det B_{N-1} + \det B_{N-1} \prod_{j=0}^{N-1} a_{jj}  +  \det A_{N-1} \prod_{j=0}^{N-1} b_{jj} \\
&\geq L_{N-1} \Big{[} \ell_{N-2} - \det A_{N-2} \det B_{N-2} + \det B_{N-2} \prod_{j=0}^{N-2} a_{jj} \\
& \qquad\qquad +  \det A_{N-2} \prod_{j=0}^{N-2} b_{jj} \Big{]}.
\end{aligned}
\end{align}
Our aim is to show that the above is equivalent to
\begin{align}\label{eq:main-ineq}
\begin{aligned}
&\ell_{N-1} -  L_{N-1}  \ell_{N-2}\\
&\geq \det B_{N-1} \bigg{(} \frac{\prod_{n=0}^{N-2}b_{nn}}{\det B_{N-2}} - 1 \bigg{)} \Big{[} a_{N-1,N-1}\det A_{N-2}  - \det A_{N-1} \Big{]} \\
& \qquad + \det A_{N-1} \bigg{(} \frac{\prod_{n=0}^{N-2}a_{nn}}{\det A_{N-2}}  - 1 \bigg{)} \Big{[} b_{N-1,N-1}\det B_{N-2}  - \det B_{N-1} \Big{]}.
\end{aligned}
\end{align}
We begin by re-arranging \eqref{eq:startingpoint} and obtain:
\begin{align*}
    \ell_{N-1} - L_{N-1}\ell_{N-2} \geq \ar,
\end{align*}
where
\begin{align*}
\ar &= \det A_{N-1} \det B_{N-1} - \det B_{N-1} \prod_{j=0}^{N-1} a_{jj}  -  \det A_{N-1} \prod_{j=0}^{N-1} b_{jj}\\
& - L_{N-1} \Big{[} \det A_{N-2} \det B_{N-2} - \det B_{N-2} \prod_{j=0}^{N-2} a_{jj}  -  \det A_{N-2} \prod_{j=0}^{N-2} b_{jj} \Big{]}.
\end{align*}
In what follows, we will use the following identities:
\begin{align}\label{eq:prod-re-arrangement}
\begin{aligned}
\prod_{j=0}^{N-1} a_{jj} 
&= a_{N-1,N-1} \prod_{j=0}^{N-2}a_{jj}\\
&= a_{N-1,N-1} \Big{[} \det A_{N-2} \bigg{(} \frac{\prod_{j=0}^{N-2}a_{jj}}{\det A_{N-2}} -1\bigg{)} + \det A_{N-2} \Big{]};\\
\prod_{j=0}^{N-1} b_{jj} 
&= b_{N-1,N-1} \prod_{j=0}^{N-2}b_{jj} \\
&= b_{N-1,N-1} \Big{[} \det B_{N-2} \bigg{(} \frac{\prod_{j=0}^{N-2}b_{jj}}{\det B_{N-2}} -1\bigg{)} + \det B_{N-2} \Big{]}.
\end{aligned}
\end{align}
We re-arrange $\ar$ in three terms:
\begin{align*}
\ar &= \Big{[} \det A_{N-1} \det B_{N-1} - L_{N-1} \det A_{N-2} \det B_{N-2} \Big{]} \\
& \qquad + \prod_{j=0}^{N-2}a_{jj} \Big{[} L_{N-1} \det B_{N-2} - a_{N-1,N-1} \det B_{N-1} \Big{]} \\
& \qquad + \prod_{j=0}^{N-2}b_{jj} \Big{[} L_{N-1} \det A_{N-2} - b_{N-1,N-1} \det A_{N-1} \Big{]}.    
\end{align*}
We simplify the second term above in the following way:
\begin{align*}
& L_{N-1} \det B_{N-2} - a_{N-1,N-1} \det B_{N-1}\\
& = \frac{\det A_{N-1}}{\det A_{N-2}}b_{N-1,N-1} \det B_{N-2} 
+ \frac{\det B_{N-1}}{\det B_{N-2}}a_{N-1,N-1}\det B_{N-2} \\
&\qquad -\frac{\det A_{N-1}}{\det A_{N-2}}\frac{\det B_{N-1}}{\det B_{N-2}}\det B_{N-2} 
- a_{N-1,N-1} \det B_{N-1}\\
& = \frac{\det A_{N-1}}{\det A_{N-2}}b_{N-1,N-1} \det B_{N-2} -\frac{\det A_{N-1} \det B_{N-1}}{\det A_{N-2}}\\
& = \frac{\det A_{N-1}}{\det A_{N-2}} \Big{(}b_{N-1,N-1} \det B_{N-2} - \det B_{N-1} \Big{)}.
\end{align*}
Similarly,
\begin{align*}
& L_{N-1} \det A_{N-2} - b_{N-1,N-1} \det A_{N-1}\\
& = \frac{\det B_{N-1}}{\det B_{N-2}} \Big{(} a_{N-1,N-1} \det A_{N-2} - \det A_{N-1} \Big{)}.
\end{align*}
Thus, we have:
\begin{align}\label{eq:R-after-starstar}
\begin{aligned}
\mathcal{R}
&=
\Big[\det A_{N-1}\det B_{N-1}-L_{N-1}\det A_{N-2}\det B_{N-2}\Big] \\
&\qquad
+ \prod_{j=0}^{N-2}a_{jj} \frac{\det A_{N-1}}{\det A_{N-2}}
\Big[ b_{N-1,N-1}\det B_{N-2}-\det B_{N-1}\Big] \\
&\qquad
+ \prod_{j=0}^{N-2}b_{jj} \frac{\det B_{N-1}}{\det B_{N-2}}
\Big[ a_{N-1,N-1}\det A_{N-2}-\det A_{N-1}\Big].
\end{aligned}
\end{align}
Using \eqref{eq:prod-re-arrangement}, we get that
\begin{align*}
&\prod_{j=0}^{N-2}a_{jj} \frac{\det A_{N-1}}{\det A_{N-2}}\Big[ b_{N-1,N-1}\det B_{N-2}-\det B_{N-1}\Big]\\
& =\det A_{N-1}\Bigg(\frac{\prod_{j=0}^{N-2}a_{jj}}{\det A_{N-2}}-1\Bigg)
\Big[b_{N-1,N-1}\det B_{N-2}-\det B_{N-1}\Big]\nonumber\\
&\quad
+\det A_{N-1}\Big[b_{N-1,N-1}\det B_{N-2}-\det B_{N-1}\Big].
\end{align*}
Similarly,
\begin{align*}
&\prod_{j=0}^{N-2}b_{jj} \frac{\det B_{N-1}}{\det B_{N-2}}\Big[ a_{N-1,N-1}\det A_{N-2}-\det A_{N-1}\Big]\\
&=\det B_{N-1}\Bigg(\frac{\prod_{j=0}^{N-2}b_{jj}}{\det B_{N-2}}-1\Bigg)
\Big[a_{N-1,N-1}\det A_{N-2}-\det A_{N-1}\Big]\nonumber\\
&\quad
+\det B_{N-1}\Big[a_{N-1,N-1}\det A_{N-2}-\det A_{N-1}\Big].
\end{align*}
Consider the sum of the second terms from each of the previous two identities and simplify:
\begin{align}\label{eq:L-iden}
\begin{aligned}
\mathcal{L}
&:= b_{N-1,N-1}\det A_{N-1}\det B_{N-2}-\det A_{N-1}\det B_{N-1} \\
&\qquad +a_{N-1,N-1}\det A_{N-2}\det B_{N-1}-\det A_{N-1}\det B_{N-1} \\
&=b_{N-1,N-1}\det A_{N-1}\det B_{N-2}
+a_{N-1,N-1}\det A_{N-2}\det B_{N-1}\\
& \qquad -2\det A_{N-1}\det B_{N-1}.
\end{aligned}
\end{align}
Using the definition of $L_{N-1}$, we have
\begin{align}\label{eq:firstterm-expanded}
\begin{aligned}
&L_{N-1}\det A_{N-2}\det B_{N-2}\\
&= \frac{\det A_{N-1}}{\det A_{N-2}}b_{N-1,N-1}\det A_{N-2}\det B_{N-2} \\
& \qquad +\frac{\det B_{N-1}}{\det B_{N-2}}a_{N-1,N-1}\det A_{N-2}\det B_{N-2} \\
&\qquad
-\frac{\det A_{N-1}}{\det A_{N-2}}\frac{\det B_{N-1}}{\det B_{N-2}}\det A_{N-2}\det B_{N-2} \\
&= b_{N-1,N-1}\det A_{N-1}\det B_{N-2}
+a_{N-1,N-1}\det A_{N-2}\det B_{N-1}\\
& \qquad -\det A_{N-1}\det B_{N-1}.
\end{aligned}
\end{align}
Therefore, the first term in $\ar$ simplifies to:
\begin{align}\label{eq:star-cancellation}
\begin{aligned}
&\det A_{N-1}\det B_{N-1}-L_{N-1}\det A_{N-2}\det B_{N-2} \\
&= 2\det A_{N-1}\det B_{N-1}
-b_{N-1,N-1}\det A_{N-1}\det B_{N-2}\\
&\qquad -a_{N-1,N-1}\det A_{N-2}\det B_{N-1}\\
& = - \mathcal{L}.
\end{aligned}
\end{align}
Combining the above, we obtain the inequality \eqref{eq:main-ineq}. Using Hadamard's inequality and \eqref{E:ineq-ann}, the right-hand side of inequality \eqref{eq:main-ineq} is nonnegative. Also, again using \eqref{E:ineq-ann}, for any $n=0,1,\dots,N-1$, the following expression is nonnegative:
\[
\frac{1}{a_{n,n}b_{n,n}}L_{n} = \frac{\det A_{n}}{a_{n,n}\det A_{n-1}} + \frac{\det B_{n}}{b_{n,n}\det B_{n-1}} -\frac{\det A_{n}}{a_{n,n}\det A_{n-1}}\frac{\det B_{n}}{b_{n,n}\det B_{n-1}}.
\]
Thus, $L_n \geq 0$ for all $n=0,1,\dots,N-1$. Finally, since $\ell_{0}=0$, using \eqref{eq:main-ineq}, the above nonnegativity and induction, we conclude that $\ell_{N-1}\geq 0$, which completes the proof.
\end{proof}

\subsection*{Acknowledgments} DG was supported by NSF grant \#235006. JM was supported by the Canada Research Chairs program and NSERC Discovery grant. PKV was supported by the Centre de recherches math\'ematiques and Laval University (CRM-Laval) Postdoctoral Fellowship.

\bibliographystyle{plain}
\bibliography{Inequality-biblio}
\end{document}